\documentclass[12pt]{article}
      \usepackage{setspace}     
     \usepackage[pdftex]{graphicx}
     \usepackage{sidecap}
\usepackage{fancyhdr, amsmath, amssymb, amsfonts, url, dsfont, mathrsfs, graphicx, epsfig,  amsthm,mathrsfs}
\usepackage{tikz}
\usepackage{wrapfig}
\usepackage{framed}
\usepackage{listings}

\usepackage{pgf}
\usepackage{tikz}
\usetikzlibrary{arrows,automata}

\usepackage{dsfont}
\usepackage{float}
\usepackage{verbatim} 
\usepackage[latin1]{inputenc}
\usepackage{cite}
\usepackage{graphicx}
\usepackage{vmargin} 
\usepackage{ifpdf} 
\usepackage{url} 
\usepackage{fancyhdr}
\usepackage{lmodern}
\usepackage{moreverb}
\usepackage{enumitem}
\usepackage{amsmath}
\usepackage{pgf}
\usepackage{tikz}
\usepackage{graphicx}
\usetikzlibrary{arrows}
\usetikzlibrary{shapes,snakes,calendar,matrix,backgrounds,folding}

\usepackage{color}
\definecolor{rltred}{rgb}{0.75,0,0}
\definecolor{rltgreen}{rgb}{0,0.5,0}
\definecolor{rltblue}{rgb}{0,0,0.75}

\usepackage{thumbpdf}
\usepackage[pdftex,colorlinks=true,urlcolor=rltred, filecolor=rltgreen, linkcolor=rltblue]{hyperref}

      \usepackage{caption}
      \usepackage{subcaption}
      \usepackage{breqn}	
      \theoremstyle{plain}
      \newtheorem{theorem}{Theorem}[section]

      \newtheorem{remark}[theorem]{Remark}

      \newtheorem{claim}[theorem]{Claim}
      \newtheorem{definition}[theorem]{Definition}

      \newcommand{\R}{{\mathbb R}}

      \newcommand{\B}{\mathcal{B}}
      \newcommand{\Z}{\mathbb{Z}}
      \newcommand{\N}{\mathbb{N}}

      \newcommand{\MM}{\mathcal{M}}

\begin{document}

\title{Escape rate for special semi-flows over non-invertible subshifts of finite type}

\author{Italo Cipriano}
\date{September 2015}

\maketitle

\begin{abstract}
We prove an escape rate result for special semi-flows over non-invertible subshifts of finite type. Our proofs are based on a discretisation of the flow and an application of an escape rate result for conformal repellers.
\end{abstract}
  
  \section{Introduction}\label{ERLD_intro}

Suppose that we have a measure preserving dynamical system. If we consider a subset of the phase space, we know that the orbit of almost every point enters it. A natural object to study in this case is the measure of the points that have not entered this subset up to a time $n\in\N.$  It is natural to think in some classical examples of uniformly hyperbolic smooth dynamical systems that this measure will decrease exponentially as $n$ increases. The escape rate through a subset of the phase space corresponds to the asymptotic rate between $n$ and the logarithm of the measure of the points that have not entered our subset up to time $n.$  Once one has understood the escape rate of a set, one may wonder how the escape rates of sets whose measure converge to zero and the measure of the sets itself are asymptotically related? The answer is that for some uniformly hyperbolic smooth dynamical systems and some particular probability measures (Gibbs measures for example) one can explicitly describe this asymptotic behaviour in the special case of the map $2x \pmod{1}$ with the Lebesgue measure \cite{Bunimovich_Yurchenko_2011}, and more generally, for conformal repellers and Gibbs measures \cite{FP}. The question that motivated this paper is: can we say something similar for smooth flows? A general answer is out of the scope of this paper, however, we will show that it is possible to obtain results for special semi-flows over non-invertible subshifts of finite type analogous to that for Gibbs measures in discrete dynamical systems.

Being more concise, suppose $\Lambda$ is a  set of possible states (a compact metric space) that evolves in time according to the transformations $\Phi^t:\Lambda\to \Lambda,$ $t\in \R^{>0}.$ If we know the state of the system at time zero, say $x\in \Lambda,$ then at time $t$ it is $\Phi^t(x).$ To be consistent we need that $\Phi^{t+s}(x)=\Phi^s(\Phi^t(x))$ for any $s,t\in \R^{>0}.$ This defines a flow $\{\Phi^t\}.$ We assume that we have an Ergodic probability measure $\nu$ on $\Lambda,$ that is an invariant probability measure in which invariant sets have either null or full measure, so that in particular $(\Lambda,\B_{\Lambda},\nu,\Phi^{t})$ is a measure preserving dynamical system, where $\B_{\Lambda}$ is the Borel algebra on $\Lambda.$ For an open set $\mathcal{H}\subset \Lambda$ and $t\in\R^{>0},$ we define  $$
 K(\nu,t,\mathcal H, \Lambda):=\log \nu \{x\in\Lambda: \Phi^sx\notin \mathcal H, s\in [0,t]\}
 $$
 and the escape rate through $\mathcal{H}$ by
 $$
 R(\nu,\mathcal H, \Lambda):=-\limsup_{t\to\infty}\frac{1}{t}K(\nu,t,\mathcal H, \Lambda).
 $$
The limit $\lim_{\nu(\mathcal{H})\to 0}\frac{R(\nu,\mathcal H, \Lambda)}{\nu(\mathcal H)}$ quantifies the asymptotic rate of the measure of the points that have not entered a subset $\mathcal{H}\subset \Lambda,$ with respect to the measure of $\mathcal{H},$ when the measure of $\mathcal{H}$ is small. Escape rates for discrete uniformly hyperbolic dynamical system are studied in \cite{Bunimovich_Yurchenko_2011}, \cite{FP} and in the references therein. In this case we have a measure preserving dynamical system $(\mathcal X,\B_{\mathcal X},\mu,T)$ and define for an open set $\mathcal{H}\subset \mathcal X, k\in\N$
 $$
 K_{\text{Discrete}}(\mu,k,\mathcal H, \mathcal X):=\log \mu \{x\in\mathcal X: T^ix\notin \mathcal H, i\in \{0,1,\ldots, k-1\} \}
 $$
 and the escape rate through $\mathcal{H}$ by
 $$
 R_{\text{Discrete}}(\mu,\mathcal H, \mathcal X):=-\limsup_{k\to\infty}\frac{1}{k}K_{\text{Discrete}}(\mu,k,\mathcal H, \mathcal X).
 $$
 In \cite{FP} they consider a discrete dynamical system $(\mathcal{X},T),$ where $(\mathcal{X},T)$ a non-invertible subshift of finite type or a conformal repeller, so in particular uniformly hyperbolic. For an equilibrium state (or Gibbs measure in this case) $\mu$ on $\mathcal{X}$ of H\"older potential $\varphi$ and pressure $P(\varphi)\in\R^{\geq 0},$ Theorem 5.1 in \cite{FP} proves that for shrinking sequences  $\{\mathcal I_n\},\mathcal I_n\subset \mathcal X$ satisfying the nested condition (Definition \ref{nested condition}) with $\cap_{n\in\N}\mathcal I_n=\{z\},z\in \mathcal X$ 
\begin{equation}\label{10_6_2015_caso_discreto}
\lim_{n\to\infty} \frac{R_{\text{Discrete}}(\mu,\mathcal{I}_n,\mathcal X)}{\mu(\mathcal{I}_n)}=\gamma  (z),
\end{equation}
where $\gamma:\mathcal{X}\to [0,1]$ is defined for $x\in\mathcal X$ by
\[
\gamma(x):= \begin{cases}
1 & \mbox{ if }x\mbox{ is not periodic},\\
1-\exp{\left(\sum_{k=0}^{p-1}\varphi(T^{k}x)-p P(\varphi)\right)} & \mbox{ if }x\mbox{ has prime period }p.\end{cases}\]
 
A special semi-flow $(\Lambda, \Phi^t)$ over a discrete and uniformly expansive dynamical system $(\mathcal{X},T),$ corresponds to the semi-flow in which every point in $\Lambda$ moves with unit speed along along the non-contracting and non-expanding direction until it reaches the boundary of $\Lambda$ and it jumps according $T.$ That is, 
for a continuous function $f:\mathcal{X}\to\R^{>0},$ we consider the continuous action $\Phi^t=\Phi_{f}^{t}$ 
  on
  \[
  \Lambda=\Lambda_{f}:=\{(x,t):x\in\mathcal X,0\leq t<f(x)\}
  \]
 onto itself defined by  
\[ 
     \Phi_{f}^{t}(x,s):= \left(T^{m}x,s+t- \sum_{k=0}^{m-1} f(T^k x)\right)\mbox{ for } \sum_{k=0}^{m-1} f(T^k x)\leq s+t<\sum_{k=0}^{m} f(T^k x),
\]
where $m\in \Z^{\geq 0}.$
The main result of this paper stablishes an analog of (\ref{10_6_2015_caso_discreto}) for a special semi-flow $(\Lambda, \Phi^t)$ over a discrete dynamical system $(\mathcal{X},T)$ and a probability measure $\mu$ on $\mathcal{X},$ where $(\mathcal{X},T)$ is a non-invertible subshift of finite type and $\mu$ is an equilibrium state of H\"older potential. We consider the invariant probability measure 
$$
\nu=\mu^f:=\frac{\mu\times Leb}{\int f d\mu}
$$
on $\Lambda,$ where $Leb$ is the Lebesgue measure on $\R.$ The goal of this paper is to prove the following theorem.

\begin{theorem}\label{theoremA}
If the roof function $f:\mathcal{X}\to \R^{>1}$ is Lipschitz, and $\{\mathcal{I}_n\},  \mathcal{I}_n\subset \mathcal{X}$ satisfies the nested condition (Definition \ref{nested condition}) with $\cap_{n\in\N}\mathcal{I}_n=\{z\}$ for $z\in\mathcal{X}.$ Then  
$$
\lim_{n\to\infty}\frac{R(\nu,\mathcal{I}_n\times \{0\},\Lambda)}{\nu(\mathcal{I}_{n}\times [0,1])}=\gamma (z).
$$
\end{theorem}

The proof follows from a discretisation of the flow and an application of (\ref{10_6_2015_caso_discreto}). This strategy was proposed by Mark Pollicott.\\

Special flows have been important in the study of hyperbolic flows. It seems natural that further generalization of our theorem will give similar results for hyperbolic flows. 
   
\section{Background}
  
We will formally introduce the definition of subshift of finite type. For notational convenience let us first denote $\langle i, j\rangle:=\{i,i+1,\ldots,j\}\subset \Z,$ for $i\in\Z, j\in\Z^{\geq i}.$ Let $A$ denote an irreducible and aperiodic $a\times a$ matrix of zeros and ones ($a\in\Z^{\geq 2}$), i.e. there exists $d\in\N$ for which $A^{d}>0$ (all coordinates of $A^d$ are strictly positive). We call the matrix $A$ transition matrix. We define the non-invertible subshift of finite type (usually called non-invertible topologically mixing subshift of finite type) $\mathcal{X}=\mathcal{X}_A \subset \langle 1,a\rangle^{\Z^{\geq 0}}$ such that
  $$\mathcal{X}:=\{(x_n)_{n=0}^{\infty}:A(x_n,x_{n+1})=1\mbox{ for all }n\in \Z^{\geq 0}\}.$$ 
  On $\mathcal{X},$ the shift $\sigma:\mathcal{X}\to \mathcal{X}$ is defined by $\sigma(x)_{n}=x_{n+1}$ for all $n\in \Z^{\geq 0}.$ For $x\in \mathcal{X}$ and $n\in \N,$ we define the cylinder $$[x]_n:=\{y\in\mathcal{X}: y_i=x_i \mbox{ for } i\in\langle 0, n-1\rangle\},$$ 
      we denote by $\xi_n$ the set of all the cylinders $[x]_n$ with $x\in\mathcal{X}$ and we call by $\B_{\mathcal{X}}$ the sigma-algebra generated by the closed sets of $\mathcal{X}$ (Borel algebra on $\mathcal{X}$). Denote by $\MM_{\sigma}$ the space of $\sigma$ invariant probability measure on $\mathcal{X},$ that is the space of probability measurs $\mu$ on $\mathcal{X}$ so that $\mu(\mathcal{A})=\mu(\sigma^{-1}\mathcal{A})$ for every measurable set $\mathcal{A}.$ For $\theta \in (0,1),$ we consider the metric on $\mathcal{X}$ given by $d_{\theta}(x,y)=\theta^{m},$ where $m=\inf\{n\in \N: x_n\neq y_n\}$ and $d(x,x)=0$ for every $x\in \mathcal{X}.$ Here $(\mathcal{X}, d_{\theta})$ is a complete metric space. We say that $f:\mathcal{X}\to\R$ is continuous if it is continuous with respect to $d_{\theta}.$ Given $f:\mathcal{X}\to\R$ continuous and $n \in \N$ define 
 \[
 V_n(f):= \sup_{z\in \mathcal{X}}\{|f(x)-f(y)|:x,y\in [z]_n\},
 \]
the Lipschitz semi-norm
\[
|f|_{\theta}:=\sup\left\{ \frac{V_n(f)}{\theta^n}:n \in\N\right\}
\]    
and the Lipschitz norm
\[
\left\Vert f\right\Vert_{\theta}:= |f|_{\theta}+\| f \|_{\infty},
\]
where $\| f \|_{\infty}:=\sup_{x\in \mathcal{X}} \{|f(x)|\}.$

The space of continuous functions with finite Lipschitz norm is called the space of Lipschitz functions (or $\theta$-Lipschitz functions) and denoted by $\mathcal{F}.$ A continuous function is $\alpha$-H\"older for $d_{\theta}$ if and only if it is Lipschitz for $d_{\theta^{\alpha}}.$

We denote the space of invariant probability measures for $(\mathcal X,\sigma)$  by $\mathcal M_{\mathcal X}.$ Given a H\"older potential $\varphi\in \mathcal{F},$ there is unique probability measure $\mu\in \mathcal M_{\mathcal X}$ that achieves the supremum 
$$\sup \{h_{\mu}(T)+\int_{\mathcal X} \varphi d\mu :\mu\in\mathcal M_{\mathcal X} \},$$
where $h_{\mu}(T)$ is the measure theoretic (or Sinai) entropy.
This probability measure $\mu$ is called equilibrium state and satisfies the Gibbs property, i.e. there exist $c_1,c_2\in\R^{>0}$ and $P\in\R^{\geq0}$ such that for every $n\in\N$
\begin{equation}\label{23_09_2015}
c_1\leq \frac{\mu([x]_n)}{\exp\left( -Pn+\sum_{i=0}^{n-1}\varphi(\sigma^i x) \right)}\leq c_2.
\end{equation}
Moreover, $\mu$ is an equilibrium state if and only if $\mu$ satisfies (\ref{23_09_2015}) for certain $c_1,c_2$ (see \cite{Rufus}), and so equilibrium states and Gibbs measures coincide.

We end this section of background with a necessary condition from \cite{FP}.
   
\begin{definition}[Nested condition]\label{nested condition}   
  We say that a family of open sets $\{\mathcal{I}_n\},\mathcal{I}_n\subset \mathcal{X}$  satisfies the \emph{nested condition} if it satisfies that:
  \begin{enumerate}
  \item each $\mathcal{I}_n$ consists of a finite union of cylinder sets, with each cylinder having length $n;$
  \item $\mathcal{I}_{n+1}\subset \mathcal{I}_n$ for every $n\in \N$ and $\cap_{n\in\N}\mathcal{I}_n=\{z\}$ for some $z\in \mathcal{X};$
  \item there exist constants $c\in \R^{>0}$ and $0<\rho<1$ such that $\mu(\mathcal{I}_n)\leq c\rho^{n}$ for all $n\in \N;$
  \item there is a sequence $\{l_n\}\subset \N$ and a constant $\kappa \in\R^{>0}$ such that $\kappa<l_n/n\leq 1$ and $\mathcal{I}_n\subset [z]_{l_n}$ for all $n\in \N;$
  \item if $\sigma^{p}(z)=z$ has prime period $p,$ then $\sigma^{-p}(\mathcal{I}_n)\cap [z]_p\subset \mathcal{I}_n$ for large enough $n.$
  \end{enumerate}
\end{definition}

\section{Proof of the Theorem}

We start this section with an easy observation.

\begin{remark}\label{rem_24_9_2015}
Given a $\theta$-Lipschitz function $f:\mathcal{X}\to \R^{>0},$ there exists $\eta:\N\to \R^{>0}$ converging to $0$ such that 
 \begin{equation}\label{opt_2_16_april_2015}
 \max \left\{ \sup_{x\in [y]_m} f(x)-\inf_{x\in [y]_m} f(x): y\in \mathcal{X} \right\} <\eta (m)
 \end{equation}
 for all $m\in \N.$ Moreover, $\eta(m)=|f|_{\theta}\theta^{m}$ for $m\in\N.$
\end{remark}

 We now proceed to the proof of Theorem \ref{theoremA}. 
       
  \begin{proof}
  We can find $\epsilon\in \R^{>0}$ such that $f>\epsilon.$ Once fixed $\epsilon,$ we can choose $\delta\in (0,\epsilon/3)$ and $m\in \N$ such that 
  \begin{equation}\label{eq1_24_09_2015}
  2\delta+\eta(m)<0.5\int fd\mu.
  \end{equation}
   This is justified by the fact that $f$ is $\theta$-Lipschitz and we can use Remark \ref{rem_24_9_2015}. We will require condition (\ref{eq1_24_09_2015}) in inequality (\ref{ineq1.2}).  
  We define an approximation of $f$ from above by
  \[ \overline{f}_{m,\delta}(x):=\left( \left[\sup_{y\in [x]_m} f(y)/\delta \right] +2 \right)\delta,\]
   and an approximation of $f$ from below by
  \[ \underline{f}_{m,\delta}(x):=\left( \left[\inf_{y\in [x]_m} f(y)/\delta \right] -2 \right)\delta.\]
  To make the notation shorter we denote $\overline{f}=\overline{f}_{m,\delta}$ and $\underline{f}=\underline{f}_{m,\delta}.$
  We consider the special flows $(\Lambda_{\overline{f}}, \Phi_{\overline{f}}^t)$ and $(\Lambda_{\underline{f}}, \Phi_{\underline{f}}^t).$ We can discretise them by considering $(\Lambda_{\overline{f}}, \Phi_{\overline{f}}^{k\delta})$ and $(\Lambda_{\underline{f}}, \Phi_{\underline{f}}^{k\delta}),$ where $k\in \Z^{\geq 0}.$ We can associate a non-invertible subshift of finite type to each discrete flow by doing the following. Define
  \[
  \mathcal{X}_{\overline{f}}:=\{(y_i)_{i=0}^{\infty}: y_i=\Phi_{\overline{f}}^{k\delta}([x]_m),x\in\mathcal{X},k\in\N,A_{\overline{f}}(y_i,y_{i+1})=1\},
  \]
  where 
\[
A_{\overline{f}}\left(\Phi_{\overline{f}}^{k\delta}([x]_{m}),\Phi_{\overline{f}}^{(k'+1)\delta}([x']_{m})\right)=
\begin{cases}
1 & \mbox{ if } C_1 \mbox{ or } C_2,\\
0 & \mbox{ if not, }\end{cases}\]
and
$$
\begin{aligned} 
& C_{1}\Leftrightarrow k=k'\mbox{ \& }x\in[x']_{m},\\
& C_{2}\Leftrightarrow\begin{cases}
(k+1)\delta=\overline{f}(x) \mbox{ \&} \\
(k'+1)\delta=\overline{f}(x') \mbox{ \&} \\
x_{i+1}=x'_i \mbox{ for all }i\in \langle 0,m-2\rangle,\end{cases}
\end{aligned}
$$
with the shift $\sigma_{\overline{f}}:\mathcal{X}_{\overline{f}}\to\mathcal{X}_{\overline{f}},$ $\sigma_{\overline{f}}:=\Phi_{\overline{f}}^{\delta}.$  We denote $\Phi_{\overline{f}}^{k\delta}([x]_m)$ by $([x]_m,k \mbox{ mod } \overline{f}([x]_m)/\delta).$

\begin{figure}[!]
    \centering
\fbox{
    \includegraphics[scale=0.4]{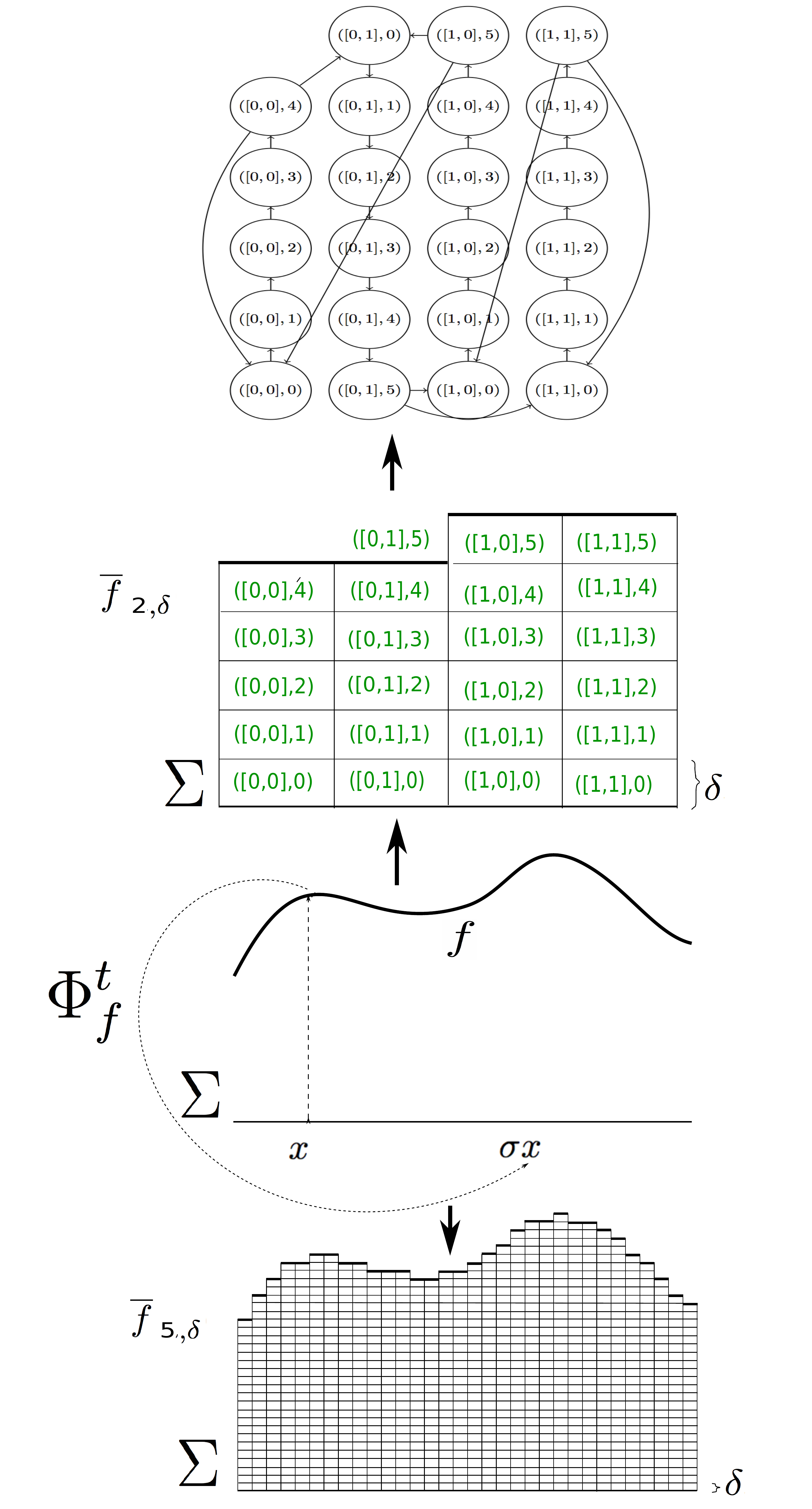}
} 
    \caption{Example of our discretisation of the flow for $\mathcal{X}=\{1,2\}^{\Z^{\geq 0}},$ where $[i,j]$ are the cylinders of length $2$ for $i,j\in\{1,2\}.$}
\end{figure}

Given a set $\mathcal{N}\subset \Z^{\geq 0}$ and $\mathcal{W}\in \xi_m$ define $$d\N|_{[0,\overline{f}(\mathcal{W})/\delta)}(\mathcal{N})=| \{n\in \mathcal{N}: n<\overline{f}(\mathcal{W})/\delta\} |,$$
the measure
\begin{equation}\label{measure_name_27_april_2015}
\tilde{\mu}^{\overline{f}}:=\frac{1}{\int \overline{f} d\mu} \sum_{\mathcal{W}\in \xi_m} \mu|_{\mathcal{W}}
\times \delta  d\N|_{[0,\overline{f}(\mathcal{W})/\delta )}
\end{equation}
 is an invariant probability measure for the subshift of finite type $( \mathcal{X}_{\overline{f}}, \sigma_{\overline{f}})$ and corresponds to the equilibrium state of a H\"older potential $\phi=\phi_{\overline{f}}:\mathcal{X}_{\overline{f}} \to \R$ (Lemma \ref{12_jun_2015_Lemma}). Again, we can do the same by replacing $\overline{f}$ with $\underline{f}.$ Notice that for a given roof function $g,$ the measure with tilde
$\tilde{\mu}^{g}$ is a discrete version of the measure $\frac{\mu\times Leb}{\int gd\mu}.$\\

Applying (\ref{10_6_2015_caso_discreto}) to the subshift of finite type that we have constructed we obtain
\[
\lim_{n\to\infty}\frac{R_{\text{Discrete}}(\tilde{\mu}^{\overline{f}},\mathcal{I}_n\times\{0\},\mathcal{X}_{\overline{f}})}{\tilde{\mu}^{\overline{f}}(\mathcal{I}_n\times\{0\})}=\gamma(z),\]
and the same can be done for $\underline{f}.$\\

By definition we have:
\begin{enumerate}
\item for $n\in\N,$ if $k\in\N$ and $\underline{f}=\underline{f}_{m,\delta},$ then 
\[K_{\text{Discrete}}(\tilde{\mu}^{\underline{f}},k,\mathcal{I}_{n} \times \{0\},\mathcal{X}_{\underline{f}})=K(\mu^{\underline{f}},\delta k,\mathcal{I}_{n} \times \{0\},\Lambda_{\underline{f}});\]
\item if $k\in\N$ and $\overline{f}=\overline{f}_{m,\delta},$ then 
\[K_{\text{Discrete}}(\tilde{\mu}^{\overline{f}},k,\mathcal{I}_{n} \times \{0\},\mathcal{X}_{\overline{f}})=K(\mu^{\overline{f}},\delta k,\mathcal{I}_{n} \times \{0\},\Lambda_{\overline{f}});\]
\item $K(\mu^{f},t,\mathcal{I}_{n} \times \{0\},\Lambda_f)$ is decreasing in $t.$
\end{enumerate}
From this we have that independently of $n\in\N$ and for any $t\in \R^{>\delta}$
\begin{equation}\label{ineq1.1}  
K_{\text{Discrete}}(\tilde{\mu}^{\underline{f}},\left\lceil t/\delta\right\rceil,\mathcal{I}_{n}\times\{0\},\mathcal{X}_{\underline{f}})\leq K(\mu^{\underline{f}},t,\mathcal{I}_{n}\times\{0\},\Lambda_{\underline{f}})
\end{equation}
and
\begin{equation}\label{ineq1.3}  
K(\mu^{\overline{f}},t,\mathcal{I}_{n}\times\{0\},\Lambda_{\overline{f}}) \leq K_{\text{Discrete}}(\tilde{\mu}^{\overline{f}},\left\lfloor t/\delta\right\rfloor,\mathcal{I}_{n}\times\{0\},\mathcal{X}_{\overline{f}}).
\end{equation}

We will need the following inequality 
\begin{equation}\label{ineq1.2}  
K(\mu^{\underline{f}},t,\mathcal{I}_{n}\times\{0\},\Lambda_{\underline{f}}) +\log \frac{1}{2} \leq K(\mu^{f},t,\mathcal{I}_{n}\times\{0\},\Lambda_{f})\leq K(\mu^{\overline{f}},t,\mathcal{I}_{n}\times\{0\},\Lambda_{\overline{f}})+\log 2.
\end{equation}

In order to prove it, let consider the inclusions 
$$
 \begin{aligned} 
&\mathcal{A}:=\left\{ (x,s')\in\Lambda_{\underline{f}}:\Phi_{\underline{f}}^{s}(x,s')\notin \mathcal{I}_{n}\times\{0\},0\leq s\leq t\right\} \cr
&\subseteq \left\{ (x,s')\in\Lambda_{f}:s'<\underline{f}(x),\Phi_{f}^{s}(x,s')\notin \mathcal{I}_{n}\times\{0\},0\leq s\leq t\right\}\cr
&\subseteq \left\{ (x,s')\in\Lambda_{f}:\Phi_{f}^{s}(x,s')\notin \mathcal{I}_{n}\times\{0\},0\leq s\leq t\right\}=:\mathcal{V}.
 \end{aligned}
 $$ 
Then $ \mu^{\underline{f}}(\mathcal{A}) \int \underline{f} d\mu\leq  \mu^f(\mathcal{V}) \int f d\mu $ and 
$\log(\mu^{\underline{f}}(\mathcal{A}))+\log(\int \underline{f} d\mu) \leq  \log( \mu^f(\mathcal{V}))+\log(\int f d\mu).$ Thus by definition $K(\mu^{\underline{f}},t,\mathcal{I}_{n} \times \{0\},\Lambda_{\underline{f}})+\log (\int \underline{f} d\mu)\leq K(\mu^f,t,\mathcal{I}_{n} \times \{0\},\Lambda_f)+\log (\int f d\mu).$ 
We chose $m$ and $\delta$ so that $2\delta+\eta(m)<0.5\int fd\mu,$ then $\underline{f}=\underline{f}_{m,\delta}$ satisfies 
\[
\int \underline{f}d\mu \geq \int f d\mu-\eta(m)-2\delta\geq \frac{\int f d\mu}{2}
\]
and $\overline{f}=\overline{f}_{m,\delta}$ satisfies 
\[
\int \overline{f}d\mu \leq \int f d\mu+\eta(m)+2\delta\leq 2\int f d\mu.
\]
From this is clear that 
\[
K(\mu^{\underline{f}},t,\mathcal{I}_{n} \times \{0\},\Lambda_{\underline{f}})+\log \frac{1}{2}\leq K(\mu^{\underline{f}},t,\mathcal{I}_{n} \times \{0\},\Lambda_{\underline{f}})+\log \frac{\int \underline{f} d\mu}{\int f d\mu}
\leq K(\mu^f,t,\mathcal{I}_{n} \times \{0\},\Lambda_f). 
\]
The second inequality in (\ref{ineq1.2}) is completely analogous, in this case we obtain 
\[K(\mu^f,t,\mathcal{I}_{n} \times \{0\},\Lambda_f)\leq
K(\mu^{\underline{f}},t,\mathcal{I}_{n} \times \{0\},\Lambda_{\underline{f}})+\log \frac{\int \overline{f} d\mu}{\int f d\mu}\leq K(\mu^{\underline{f}},t,\mathcal{I}_{n} \times \{0\},\Lambda_{\underline{f}})+\log 2.
\]

In the next inequality we will use this identity: 
\[
\tilde{\mu}^{\overline{f}}(\mathcal{I}_n\times \{0\})= \frac{\delta}{\int \overline{f} d\mu} \mu(\mathcal{I}_n).
\]
For all $t\in\R^{>\delta},$
$$
\begin{aligned}
&\frac{1}{\mu(\mathcal{I}_{n})}\frac{1}{t}K(\mu^{f},t,\mathcal{I}_{n}\times\{0\},\Lambda_f)\cr
&\leq \frac{1}{\mu(\mathcal{I}_{n})}\frac{1}{t}K_{\text{Discrete}}(\tilde{\mu}^{\overline{f}},\left\lfloor t/\delta\right\rfloor,\mathcal{I}_{n}\times\{0\},\mathcal{X}_{\overline{f}})+\frac{\log2}{\mu(\mathcal{I}_{n})t}\cr
&\leq\frac{\left\lfloor t/\delta\right\rfloor }{\left[ t/\delta\right]}\frac{1}{\tilde{\mu}^{\overline{f}}(\mathcal{I}_{n}\times\{0\})\int\overline{f}d\mu}\frac{1}{\left\lfloor t/\delta\right\rfloor }K_{\text{Discrete}}(\tilde{\mu}^{\overline{f}},\left\lfloor t/\delta\right\rfloor,\mathcal{I}_{n}\times\{0\},\mathcal{X}_{\overline{f}})+\frac{\log2}{\mu(\mathcal{I}_{n})t}.
\end{aligned}
$$
In the last inequality, taking $\limsup_{t\to\infty}$ on both sides, then letting $n$ tend to infinity, and finally multiplying by $-1,$ allows us to write
\begin{equation}\label{eq1} 
\lim_{n\to\infty}\frac{ R(\mu^f,\mathcal{I}_n\times \{0\},\Lambda_f)}{\mu(\mathcal{I}_{n})}\geq \gamma(z) \frac{1} {\int\overline{f}d\mu}.    
\end{equation}
Similarly we obtain
\begin{equation}\label{eq2} 
\lim_{n\to\infty}\frac{R(\mu^f,\mathcal{I}_n\times \{0\},\Lambda_f)} {\mu(\mathcal{I}_{n})} \leq \gamma(z) \frac{1} {\int\underline{f}d\mu}.   
\end{equation}
Taking $f^*=\overline{f}$ or $\underline{f},$ by definition we have 
\begin{equation}\label{ineq2}
\int |f-f^*|d\mu\leq 2\delta+\eta(m). 
\end{equation}
This combined with inequalities (\ref{eq1}),(\ref{eq2}) and the fact that $\delta$ can be taken arbitrarily small, and $m$ arbitrarily large concludes the proof.
   \end{proof}
 
\section{Appendix}
 
 The following claim used in the proof of Theorem \ref{theoremA} is well known, however we include a demonstration for completeness.
 
\begin{claim}\label{12_jun_2015_Lemma}
The probability measure
\[
\tilde{\mu}^{\overline{f}}:=\frac{1}{\int \overline{f} d\mu} \sum_{\mathcal{W}\in \xi_m} \mu|_{\mathcal{W}}
\times \delta  d\N|_{[0,\overline{f}(\mathcal{W})/\delta )}
\]
 is an invariant probability measure for the subshift of finite type $( \mathcal{X}_{\overline{f}}, \sigma_{\overline{f}})$ and corresponds to the equilibrium state of a H\"older potential $\phi=\phi_{\overline{f}}:\mathcal{X}_{\overline{f}} \to \R.$
\end{claim}

\begin{proof}
We introduce some notation. For $n\in \N$ we define the set of allowed words of length $n,$ $\mathcal{X}_{n}:=\{x_{[0,n)}:=x_0x_1\ldots x_{n-1} :x\in\mathcal{X}\}.$
In what follows we take $m$ and $\delta$ fixed in the definition of $f^*=\overline{f}$ or $\underline{f}.$
We define the function $\tilde{\pi}=\tilde{\pi}_{m,\delta}:\mathcal{X}_{f^*}\to \sigma$ so that the image of $$\overline{x}=(x_0,l_0),(x_1,l_1),\ldots, (x_n,l_n), \ldots$$ is given by $\tilde{\pi}(\overline{x})=x_{i_0} x_{i_1} x_{i_2}\ldots$ where $i_0=0$ and for $n\in\N,$ $i_n=\min\{k\in\Z^{>i_{n-1}}:l_k=0\}.$ We extend the definition of $\tilde{\pi}$ to the case of finite sequences and given 
 $$
 w=\overline{x}_{[0,k)}=(x_0,l_0),\ldots,(x_{k-1},l_{k-1})
 $$
 for some $k\in\N$ where $\overline{x}\in \mathcal{X}_{f^*},$ we define $\#w:=|\{n\in \langle 1, k-1 \rangle:l_n=0\}|+1.$\\
 
By definition, given $\overline{x}\in \mathcal{X}_{f^*}, i,j,k\in \Z^{\geq 0}$ with $i<j$ and $w=\overline{x}_{[i,j)}$ we have that 
$$
\tilde{\mu}^{f^*}([w]_k^{k+j-i})=\frac{\delta}{\int f^* d\mu}\mu([\tilde{\pi}(w)]_{\#w+1}).
$$
This can be seen as an alternative way to write the same measure defined in (\ref{measure_name_27_april_2015}). 
  
For $f^*=\overline{f}$ or $\underline{f}$ we need to prove that the measure
\[
\tilde{\mu}^{f^*}=\frac{1}{\int f^* d\mu} \sum_{w\in \xi_m} \mu|_{w}
\times \delta  d\N|_{[0,f^*(w)/\delta )} \] is an invariant probability measure for the subshift of finite type $( \mathcal{X}_{f^*}, \sigma_{f^*})$ and corresponds to the equilibrium state of a H\"older potential.\\

From a corollary of the Kolmogorov consistency theorem on a subshift of finite type $\mathcal{X}\subset \langle 1,a \rangle^{\Z_+},$ where $a\in\Z^{\geq 2},$ the set of $\sigma$ invariant probability measure is identified one-to-one with the set of maps $\mu:\B_{\mathcal{X}}\to\R^{>0}\cup\{0,\infty\}$ such that 
\begin{equation}\label{21_april_2015_1}
\sum_{s=1}^{a}\mu([s]_1)=1
\end{equation}
and for all $x\in\mathcal{X},$ for all integers $i,j,k\in \Z^{\geq 0}$ with $i<j$ we have for $w=x_{[i,j)}$
\begin{equation}\label{21_april_2015_2}
\mu ([w]_k^{k+j-i})=\sum_{s=1}^{a} \mu([w,s]_k^{k+j-i+1})
\end{equation} 
and 
\begin{equation}\label{21_april_2015_3}
\mu ([w]_{k+1}^{k+j-i+1})=\sum_{s=1}^{a} \mu ([s,w]_k^{k+j-i+1}).
\end{equation} 

In what follows let consider $m$ and $\delta$ fixed in the definition of $f^*=\overline{f}$ or $\underline{f}.$\\

Then, for the first part of the proof we need to check that $\tilde{\mu}^{f^*}$ satisfies (\ref{21_april_2015_1}), (\ref{21_april_2015_2}) and (\ref{21_april_2015_3}). 
We start by proving that $\tilde{\mu}^{f^*}$ satisfies (\ref{21_april_2015_1}), indeed
$$
 \begin{aligned} 
\sum_{[\overline{y}]_1:\overline{y}\in \mathcal{X}_{f^*}} \tilde{\mu}^{f^*}([\overline{y}]_1) &= \sum_{C\in \xi_m}\sum_{i=0}^{f^*(C)/\delta-1} \frac{\delta \mu(C)}{\int f^* d\mu}=\frac{1}{\int f^* d\mu}\sum_{C\in \xi_m}\mu(C)f^*(C)=1.
 \end{aligned} 
$$
In order to prove (\ref{21_april_2015_2}) and (\ref{21_april_2015_3}), let us suppose that $\overline{x}\in\mathcal{X}_{f^*}$ and $i,j,k\in\Z^{\geq 0}$ with $i<j.$ Denote $w=\overline{x}_{[i,j)}$ where 
$\overline{x}_{[i,j)}=(x_0,l_0),\ldots,(x_{j-1},l_{j-1}),$ and for shorter notation define also $v(x)=\frac{f(x)}{\delta}$  for $x\in \mathcal{X}_m.$ We have that:
$$
 \begin{aligned} 
&\sum_{ \begin{subarray}{c} (x,l):\\ x\in\mathcal{X}_m,l\in [0,v(x)) \end{subarray}} \tilde{\mu}^{f^*}\left(\left[w,(x,l)\right]_k^{j-i+k+1}\right) \cr
&=\begin{cases}
\sum_{C \in \xi_m}\frac{\delta}{\int f^* d\mu}\mu\left([\tilde{\pi}(w),C]_{k}^{k+\#w+2}\right) &\mbox{ if }1+l_{j-1}=v(x_{j-1})\\
\frac{\delta}{\int f^* d\mu}\mu( [\tilde{\pi}(w)]_{\#w+1}) &\mbox{ otherwise } 
\end{cases}\\
&=\frac{\delta}{\int f^* d\mu}\mu([\tilde{\pi}(w)]_{\#w+1})\cr
&=\tilde{\mu}^{f^*}([w]_k^{k+j-i})
 \end{aligned} 
$$
from which (\ref{21_april_2015_2}) follows; and
$$
 \begin{aligned} 
\sum_{\begin{subarray}{c} (x,l):\\ x\in\mathcal{X}_m,l\in [0,v(x)) \end{subarray}} \tilde{\mu}^{f^*}\left(\left[(x,l),w\right]_k^{j-i+k+1}\right)
&=\sum_{x\in\mathcal{X}_m} \frac{\delta}{\int f^* d\mu} \mu([x,\tilde{\pi}(w)]_k^{k+\#w+2})\cr 
&=\frac{\delta}{\int f^* d\mu}\mu([\tilde{\pi}(w)]_k^{k+\#w+1})\cr
&= \tilde{\mu}^{f^*}([w]_k^{j-i+k})
 \end{aligned} 
$$
hence (\ref{21_april_2015_3}).\\

To prove that $\tilde{\mu}^{f^*}$ is the equilibrium state of a H\"older potential we will find explicitly a H\"older potential $\tilde{\varphi}=\tilde{\varphi}_{m,\delta}$ associated to it. Suppose that $\mu$ is the equilibrium state of an $\alpha$-H\"older potential $\varphi,$ then the candidate is $\tilde{\varphi}(\overline{x})=\tilde{\varphi}(\tilde{\pi}(\overline{x})).$\\

We observe that $d(\overline{x},\overline{y})\leq \theta ^{k\left\Vert f\right\Vert/\delta}$ implies 
$d(\tilde{\pi}(\overline{x}),\tilde{\pi}(\overline{y}))\leq \theta ^{mk}$ and 
\[
d(\tilde{\pi}(\overline{x}),\tilde{\pi}(\overline{y}))^{\left\Vert f\right\Vert/\delta}\leq d(\overline{x},\overline{y})^m.
\]
Therefore
\[
\sup_{\overline{x}\neq \overline{y}}\frac{d\left(\tilde{\varphi}(\overline{x}),\tilde{\varphi}(\overline{y}) \right)}{d(\overline{x},\overline{y})^{\alpha \delta m/\left\Vert f\right\Vert}} \leq \sup_{\overline{x}\neq \overline{y}}\frac{d\left(\varphi(\tilde{\pi}(\overline{x})),\varphi(\tilde{\pi}(\overline{y})) \right)}{d(\tilde{\pi}(\overline{x}),\tilde{\pi}(\overline{y}))^{\alpha}}<\infty
\]
because we assumed that $\varphi$ is $\alpha$-H\"older. This proves that $\tilde{\varphi}$ is $\frac{\alpha \delta m}{\left\Vert f\right\Vert}$-H\"older.

To prove that $\tilde{\mu}^{f^*}$ is an equilibrium state it is enough (see \cite{Rufus}) to check that it is Gibbs, i.e. it satisfies (\ref{23_09_2015}). For notational convenience, given $n\in\N$ a subshift of finite type $(\mathcal Y,\sigma)$ and a function $\phi:\mathcal Y\to\R,$ let us denote $S^{\sigma}_n\varphi(y)=\sum_{k=0}^{n-1}\varphi(\sigma^k y).$
Suppose $m\delta /\left\Vert f\right\Vert=1,$ and for notational convenience call $s=\left\lfloor m\delta k/\left\Vert f\right\Vert\right\rfloor.$ We have the following bounds:
$$
 \begin{aligned} 
&\sup_{\overline{x}\in \mathcal{X}_{f^*}}\frac{\tilde{\mu}^{f^*} \left([\overline{x}_{[0,k)}]_k\right)}{\exp\{-Pk+S^{\sigma}_k \tilde{\varphi}(\overline{x})\}}\cr
&\leq \frac{\delta}{\int f^* d\mu} \sup_{x\in\mathcal{X}}\frac{\mu ( [ x_{[0,s+1)} ]_{s+1})}{\exp\{-P s/[m\delta /\left\Vert f\right\Vert ]+S^{\sigma}_{s}\varphi(x) /[m\delta /\left\Vert f\right\Vert ] \}}\cr
&=\frac{\delta}{\int f^* d\mu} \sup_{x\in\mathcal{X}}\frac{\mu ([x_{[0,s+1)}]_{s+1})}{\exp\{-Ps+S^{\sigma}_{s}\varphi(x) \}}\cr
&\leq \frac{\delta c_2}{\int f^* d\mu}, 
 \end{aligned} 
$$
and
$$
 \begin{aligned} 
\sup_{\overline{x}\in \mathcal{X}_{f^*}} \frac{\tilde{\mu}^{f^*} \left([ \overline{x}_{[0,k)} ]_{k}\right) }{\exp\{-Pk+S^{\sigma}_k \tilde{\phi}(\overline{x})\}}&\geq \frac{\delta}{\int f^* d\mu} \sup_{x\in\mathcal{X}}\frac{\mu\left([x_{[0,k)}]_k\right)}{\exp\{-P k+S^{\sigma}_k\varphi(x) \}}\cr
&\geq \frac{\delta c_1}{\int f^* d\mu}. 
 \end{aligned} 
$$
This concludes the demonstration.
\end{proof}

\cleardoublepage

\end{document}